\documentclass[12pt]{amsart}
\usepackage{amsmath,amsthm,amsfonts,amssymb,latexsym}
\usepackage{tikz-cd}
\usepackage{hyperref}
\usepackage{ulem}
\usepackage{pgf,tikz}
\usepackage{mathrsfs}
\usetikzlibrary{arrows}

\usepackage{enumerate}
\usepackage[shortlabels]{enumitem}
\usepackage{color}
\usepackage{xcolor} 
\begin{document}

\theoremstyle{plain}

\newtheorem{remark}{Remark}
\newtheorem{thm}{Theorem}[section]
\newtheorem{rem}[thm]{Remark}
\newtheorem{corollary}[thm]{Corollary}
\newtheorem{lemma}[thm]{Lemma}
\newtheorem{proposition}[thm]{Proposition}
\theoremstyle{definition}
\newtheorem*{defn}{Definition}
\newtheorem*{ThmA}{Theorem A}
\newtheorem*{ThmB}{Theorem B}
\newtheorem*{ThmC}{Theorem C}
\newtheorem*{PropB}{Proposition B}
\newtheorem*{CorB}{Corollary B}
\newtheorem*{CorC}{Corollary C}
\newtheorem*{CorD}{Corollary D}
\newtheorem*{ConjectureB}{Conjecture B}
\newenvironment{enumeratei}{\begin{enumerate}[\upshape (a)]}
    {\end{enumerate}}

\newenvironment{Enumeratei}{\begin{enumerate}[\upshape (A)]}
    {\end{enumerate}}

\def\irr#1{{\rm Irr}(#1)}
\def\cent#1#2{{\bf C}_{#1}(#2)}
\def\pow#1#2{{\mathcal{P}}_{#1}(#2)}
\def\syl#1#2{{\rm Syl}_#1(#2)}
\def\hall#1#2{{\rm Hall}_#1(#2)}
\def\nor{\trianglelefteq\,}
\def\oh#1#2{{\bf O}_{#1}(#2)}
\def\zent#1{{\bf Z}(#1)}
\def\sbs{\subseteq}
\def\gen#1{\langle#1\rangle}
\def\aut#1{{\rm Aut}(#1)}
\def\out#1{{\rm Out}(#1)}
\def\gv#1{{\rm Van}(#1)}
\def\fit#1{{\bf F}(#1)}
\def\frat#1{{\bf \Phi}(#1)}
\def\gammav#1{{\Gamma}(#1)}
\newcommand{\p}{{\mathbb P}}
\newcommand{\N}{{\mathbb N}}
\newcommand{\F}{{\mathbb F}}
\def\fitd#1{{\bf F}_{2}(#1)}
\def\irr#1{{\rm Irr}(#1)}
\def\dl#1{{\rm dl}(#1)}
\def\h#1{{\rm h}(#1)}
\def\ibr#1#2{{\rm IBr}_#1(#2)}
\def\cs#1{{\rm cs}(#1)}
\def\m#1{{\rm m}(#1)}
\def\n#1{{\rm n}(#1)}
\def\cent#1#2{{\bf C}_{#1}(#2)}
\def\hall#1#2{{\rm Hall}_#1(#2)}
\def\syl#1#2{{\rm Syl}_#1(#2)}
\def\nor{\trianglelefteq\,}
\def\norm#1#2{{\bf N}_{#1}(#2)}
\def\oh#1#2{{\bf O}_{#1}(#2)}
\def\Oh#1#2{{\bf O}^{#1}(#2)}
\def\zent#1{{\bf Z}(#1)}
\def\sbs{\subseteq}
\def\gen#1{\langle#1\rangle}
\def\aut#1{{\rm Aut}(#1)}
\def\gal#1{{\rm Gal}(#1)}
\def\alt#1{{\rm Alt}(#1)}
\def\sym#1{{\rm Sym}(#1)}
\def\out#1{{\rm Out}(#1)}
\def\gv#1{{\rm Van}(#1)}
\def\fit#1{{\bf F}(#1)}
\def\lay#1{{\bf E}(#1)}
\def\fitg#1{{\bf F^*}(#1)}

\def\GF#1{{\rm GF}(#1)}
\def\SL#1{{\rm SL}_{2}(#1)}
\def\PSL#1{{\rm PSL}_{2}(#1)}

\def\gammav#1{{\Gamma}(#1)}
\def\V#1{{\rm V}(#1)}
\def\E#1{{\rm E}(#1)}
\def\b#1{\overline{#1}}

 \def\sl#1#2{{\rm SL}_{#1}(#2)}
 \def\gl#1#2{{\rm GL}_{#1}(#2)}
 \def\cl#1#2{{\rm cl}_{#1}(#2)}
\def\Z{{\mathbb{Z}}}
\def\C{{\Bbb C}}
\def\Q{{\Bbb Q}}
\def\inv{^{-1}}
\def\irr#1{{\rm Irr}(#1)}
\def\irrv#1{{\rm Irr}_{\Bbb R}(#1)}
 \def\irrk#1{{\rm Irr}_{ {\rm rv}, K}(#1)}
 \def\irrc#1{{\rm Irr}_{C}(#1)}
  \def\irrf#1{{\rm Irr}_{\mathfrak{F}'}(#1)}
   \def\ext#1{{\rm Ext}(#1)}
   \def\irrh#1{{\rm Irr}_{H}(#1)}
  \def\re#1{{\rm Re}(#1)}
  \def\csrv#1{{\rm cs}_{\rm rv}(#1)}
   \def\clrv#1{{\rm Cl}_{\rm rv}(#1)}
     \def\clk#1{{\rm Cl}_{{\rm rv}, K}(#1)}
  \def\bip#1{{\rm B}_{p'}(#1)}
  \def\irra#1{{\rm Irr}_A(#1)}
   \def\irrs#1{{\rm Irr}_\sigma(#1)}
   \def\irrp#1{{\rm Irr}_{p'}(#1)}
 \def\cdrv#1{{\rm cd}_{\rm rv}(#1)}
 \def\bip#1{{\rm B}_{p'}(#1)}
\def\cdrv#1{{\rm cd}_{\rm rv}(#1)}
\def\cd#1{{\rm cd}(#1)}
\def\irrat#1{{\rm Irr}_{\rm rat}(#1)}
\def\cdrat#1{{\rm cd}_{\rm rat}(#1)}
\def \c#1{{\cal #1}}
\def\cent#1#2{{\bf C}_{#1}(#2)}
\def\syl#1#2{{\rm Syl}_#1(#2)}
\def\oh#1#2{{\bf O}_{#1}(#2)}
\def\Oh#1#2{{\bf O}^{#1}(#2)}
\def\zent#1{{\bf Z}(#1)}
\def\det#1{{\rm det}(#1)}
\def\ker#1{{\rm ker}(#1)}
\def\norm#1#2{{\bf N}_{#1}(#2)}
\def\alt#1{{\rm Alt}(#1)}
\def\iitem#1{\goodbreak\par\noindent{\bf #1}}
    \def \mod#1{\, {\rm mod} \, #1 \, }
\def\sbs{\subseteq}

\setlist[itemize]{font=\color{itemizecolor}}
\colorlet{itemizecolor}{.}
\setlist[enumerate]{font=\color{enumeratecolor}}
\colorlet{enumeratecolor}{.}
\setlist[description]{font=\bfseries\color{descriptioncolor}}
\colorlet{descriptioncolor}{.}

\def\cE{\bar{\rm E}}

\def \nq{\mathfrak{N}_q}

\marginparsep-0.5cm

\renewcommand{\thefootnote}{\fnsymbol{footnote}}
\footnotesep6.5pt

\title[On a Character Correspondence associated to $\mathfrak{F}$-projectors]{On a Character Correspondence associated to $\mathfrak{F}$-projectors}

\begin{abstract}
We study the conditions under which the head characters of a finite solvable group, as defined by I. M. Isaacs, behave well with respect to restriction. We also determine the intersection of the kernels of all head characters of the group. Using G. Navarro's definition of $\mathfrak{F}'$-characters, we generalize these results for any saturated formation $\mathfrak{F}$ containing the formation of nilpotent groups.
\end{abstract}

\author[M. J. Felipe]{Mar\'{\i}a Jos\'e Felipe}
\address{Mar\'{\i}a Jos\'e Felipe, Institut Universitari de Matem\`atica Pura i Aplicada, \newline
Universitat Polit\`ecnica de Val\`encia, 
Val\`encia, Spain.}
\email{mfelipe@mat.upv.es}

\author[I. Gilabert ]{Iris Gilabert}
\address{Iris Gilabert, Institut Universitari de Matem\`atica Pura i Aplicada, \newline
Universitat Polit\`ecnica de Val\`encia, 
Val\`encia, Spain.}
\email{igilman@posgrado.upv.es}

\author[L. Sanus]{Lucia Sanus}
\address{Lucia Sanus, Departament de Matem\`atiques, Facultat de
 Matem\`atiques, \newline
Universitat de Val\`encia,
46100 Burjassot, Val\`encia, Spain.}
\email{lucia.sanus@uv.es}



\thanks{ This research is partially supported by the Generalitat Valenciana (CIAICO/2021/163). The second author is  supported by a grant (PAID-01-23 funded by the Universitat Polit\`ecnica de Val\`encia and subsequently CIACIF/2023/389 funded by the Generalitat Valenciana). The second and third authors are partially supported by the
Spanish Ministerio de Ciencia e Innovaci\'on
(PID2022-137612NB-I00 funded by MCIN/AEI/10.13039/501100011033 and `ERDF A way of making Europe'). The first and second authors are partially supported by Ayuda a Primeros Proyectos de Investigaci\'on (PAID-06-23) from Vicerrectorado de Investigaci\'on de la Universitat Polit\`ecnica de Val\`encia (UPV)}

\thanks{We would like to thank Gabriel Navarro for bringing this problem to our attention and for useful conversations on the subject. We would also like to thank the referee for their helpful comments and suggestions.}

\dedicatory{Dedicated to the memory of I. M. Isaacs.}

\maketitle

\section{Introduction}

All groups considered in this work are finite. In \cite{Isaacs}, I. M. Isaacs constructed a canonical subset of the complex irreducible characters of a solvable group $G$ associated to the linear characters of a Carter subgroup $C$ of $G$. Recall that a Carter subgroup of $G$ is a self-normalizing nilpotent subgroup of $G$ (see \cite{C}). Isaacs called these the {\it head characters} of $G$, and proved that the number of those was $|C/C'|$, the number of linear characters of $C$.

However, not many general properties of the head characters are known. In this work, we prove the following.

\begin{ThmA} Let $G$ a finite solvable group, let $C$ be a Carter subgroup of $G$ and let $\chi \in \irr G$ be a head character of $G$. Let $N$ be any normal subgroup of $G$. Then the following hold.
\begin{enumerate}
    \item[(a)] The restriction $\chi_N$ contains a unique $C$-invariant irreducible character $\theta$.
    \item[(b)] The restriction $\chi_{NC}$ contains a head character $\gamma$ of $NC$. Furthermore, $\gamma_N=\theta$. Hence, any other head character of $NC$ contained in $\chi_{NC}$ is of the form $\lambda\gamma$, for some linear $\lambda \in \irr {NC/N}$.
    \item[(c)] We have that $\gamma(1)$ divides $\chi(1)$ and that $\chi(1)/\gamma(1)$ divides $|G:NC|$.
\end{enumerate}
\end{ThmA}

We remark from part (c),  that if $G/N$ is nilpotent, then $\chi(1) = \gamma(1)$, and therefore $\chi_N$ is irreducible. On the other hand, if $N$ is the trivial group, part (c) tells us that $\chi(1)$ divides $|G : C|$. 
That is, part (c) proves at the same time both properties that Isaacs establishes in Theorem A of \cite{Isaacs}.

In Theorem B, we turn our attention to the kernels of the head characters and prove the following.

\begin{ThmB} Let $G$ be a finite solvable group and $C$ be a Carter subgroup of $G$. The intersection of the kernels of the head characters of $G$ is the largest normal subgroup
$N$ of $G$ such that $N \cap C$ is contained in $C'$.
\end{ThmB}

Theorem B is the exact analog of an unpublished result of Navarro which we present here with 
his permission.

\begin{ThmC}[Navarro] Let $p$ be a prime, let $G$ be a finite $p$-solvable group and let $P$ be a Sylow $p$-subgroup of $G$. The intersection of the kernels of the characters of $G$ of degree not divisible by $p$ is the largest normal subgroup $N$ of $G$ such that $\norm N P$ is contained in $P'$.
\end{ThmC}

What is the relationship between the head characters and the $p'$-degree characters of a solvable group? If $\mathfrak{F}$ is a saturated formation and $G$ is a solvable group, Navarro defined in \cite{N} a canonical subset $\irrf G$  of the irreducible characters of $G$ of size $|\irr{N_G(H)/H'}|$, where $H$ is an $\mathfrak{F}$-projector of $G$.
If $\mathfrak{F}$ is the formation of nilpotent groups, then $H$ is a Carter subgroup,
and he proved that $\irrf G$ are the head characters of $G$; in particular, reproving Isaacs result. If $\mathfrak{F}$ is the formation of $p$-groups, then $H$ is a Sylow $p$-subgroup
and this reproved the McKay conjecture for solvable groups. Particularly, if a Carter subgroup of $G$ is a Sylow $p$-subgroup of $G$, then the head characters of $G$ are exactly the irreducible characters of $G$ of degree not divisible by $p$.

In fact, our results are far more general. When $\mathfrak{F}$ is a saturated formation containing the class of nilpotent groups, we shall prove convenient versions of Theorems A and B for $\irrf G$ (see respectively Theorems \ref{th A general} and \ref{kernels}) at the same time that we characterize Navarro's $\mathfrak{F}'$-characters (see Theorem \ref{caracterización descendente}).

In the literature, we may find  similar correspondences to the one we study. For instance, E. C. Dade and D. Gajendragadkar constructed in \cite{Dade,Gajendragadkar}  another subset of $\irr G$, this time associated to the system normalizers of a solvable group $G$. This subset was described by Gajendragadkar from the better known \textit{fully factorable characters}. In the same way we do, they worked in the more general context of saturated formations containing the class of nilpotent groups.

There are still several questions on the head characters and Carter subgroups
that remain open. Among them,  if the character table of $G$ determines the set of head characters.
We are confident that our results might be helpful in solving these and other related questions.

\section{Some properties on projectors}

Recall that a class of groups $\mathfrak{F}$ is called a \textit{formation} if it is closed under quotients and satisfies the following property: whenever $G/N, G/M \in \mathfrak{F}$ for some  finite group $G$, it follows that $G/(M \cap N) \in \mathfrak{F}$. A formation $\mathfrak{F}$ is said to be \textit{saturated} if $G \in \mathfrak{F}$ if and only if $G/\Phi(G) \in \mathfrak{F}$, where $\Phi(G)$ denotes the Frattini subgroup of $G$. Throughout the remainder of the section, we will denote by $\mathfrak{F}$ a saturated formation. Examples of saturated formations are the classes of finite $p$-groups with $p$ a prime, finite $\pi$-groups with $\pi$ a set of primes, finite nilpotent groups, and finite supersolvable groups. It is well known, however, that abelian groups form a formation that is not saturated. As basic references on this topic, see \cite{Doerk} and \cite{Rob}.

The notion of formations was first introduced by Gaschütz in \cite{Gas1} as a context where both definitions of Carter and Sylow subgroups could be unified under the more general concept of $\mathfrak{F}$-covering subgroups. Gaschütz also proved that, when $\mathfrak{F}$ is saturated and $G$ is solvable, these subgroups coincide with \textit{$\mathfrak{F}$-projectors}, which are subgroups characterized by the following Lemma (see 9.5.6 in \cite{Rob}).

\begin{lemma}
    Let $\mathfrak{F}$ be a saturated formation and $G$ a solvable group. There exists exactly one conjugacy class of subgroups $H$ of $G$ such that, for every normal subgroup $N$ of $G$, the subgroup $HN/N$ satisfies the following properties:
\begin{enumerate}
    \item[(a)] $HN/N \in \mathfrak{F}$;
    \item[(b)] if $HN/N$ is a subgroup of the subgroup $V/N$ of $G/N$ with $V/N \in \mathfrak{F}$, then $HN/N = V/N$.
\end{enumerate}
In other words, $HN/N$ is $\mathfrak{F}$-maximal in $G/N$ for all normal subgroups $N$ of $G$.
\end{lemma}

As previously mentioned, Carter subgroups are the projectors with respect to the saturated formation of nilpotent groups, which we will denote by $\mathfrak{N}$ from now on. Similarly, if $p$ is a prime number and $\pi$ is a set of prime numbers, then $p$-Sylow subgroups and $\pi$-Hall subgroups are the $\mathfrak{F}$-projectors for the saturated formations of $p$-groups and $\pi$-groups, respectively.

Listed below are some fundamental properties of $\mathfrak{F}$-projectors which we will need in our proofs.

\begin{proposition} \label{basics}Let $\mathfrak{F}$ be a saturated formation and $G$ be a solvable group. Suppose $H$ is an $\mathfrak{F}$-projector of $G$. Then
\begin{itemize}
    \item[(a)] if $U$ is a subgroup of $G$ containing $H$, then $H$ is also an $\mathfrak{F}$-projector of $U$.
\end{itemize}
Moreover, if $N$ is a normal subgroup of $G$, then
\begin{itemize}
    \item[(b)] $HN/N$ is an $\mathfrak{F}$-projector of $G/N$;
    \item[(c)] if $U/N$ is an $\mathfrak{F}$-projector of $G/N$, then $U=H^gN$ for some $g \in G$;
    \item[(d)] $\norm {G}{NH} = N \norm G H$.
\end{itemize}
Finally, if $\mathfrak{N} \subseteq \mathfrak{F}$, that is, if the class of nilpotent groups is contained in $\mathfrak{F}$, then
\begin{itemize}
    \item[(e)] if $U$ is a subgroup of $G$ containing $H$, then $\norm G U = U$. In particular, $H$ is self-normalizing.
\end{itemize}
\end{proposition}

\begin{proof}
Assertions (a) to (c) are proved in Section 9.5 of \cite{Rob}. If $g \in \norm{G}{HN}$, then both $H$ and $H^g$ are $\mathfrak{F}$-projectors of $HN=H^gN$. Using (c), there exists $n \in HN$ which can be assumed to be in $N$ and such that $H^n=H^g$. Then $gn^{-1} \in \norm{G}{H}$. Thus $\norm{G}{HN} \leq N \norm{G}{H}$. Since the converse is trivial, this proves (d).

Now suppose $\mathfrak{N} \subseteq \mathfrak{F}$ and let $H$ be an $\mathfrak{F}$-projector of a solvable group $G$. Let $U$ be a subgroup of $G$ containing $H$. By (a), $H$ is also an $\mathfrak{F}$-projector of $\norm G U$, and thus by (b) we have that $U/U$ is an $\mathfrak{F}$-projector of $\norm G U /U$. By definition of projectors, this means that $\{1\}$ is $\mathfrak{F}$-maximal in $\norm G U /U$. If $U\neq \norm G U$, then there exists some nontrivial element $x \in \norm G U/U$ and the subgroup $\langle x \rangle$ being nilpotent (and thus in $\mathfrak{F}$) contradicts the $\mathfrak{F}$-maximality of $\{1\}$. This proves (e).
\end{proof}

In further sections, it will be of particular importance to assume $\mathfrak{N}\subseteq \mathfrak{F}$ in order to ensure, by Proposition \ref{basics}(e), that the $\mathfrak{F}$-projectors are self-normalizing. Note, for now, that there exist saturated formations both satisfying and not satisfying this hypothesis. For instance, the previously mentioned saturated formations of $\pi$-groups (if $\pi$ is a finite set of primes) do not contain all nilpotent groups. However, the classes of supersolvable groups, solvable groups of nilpotent length at most $l$ (in particular, meta-nilpotent groups), solvable $p$-nilpotent groups or $p$-decomposable groups, for some prime $p$, are examples of saturated formations containing the nilpotent groups (see IV.3 in \cite{Doerk}).

\medskip

The \textit{$\mathfrak{F}$-residual} $G^{\mathfrak{F}}$ of $G$ is its smallest normal subgroup such that the quotient $G/G^{\mathfrak{F}}$ lies in $\mathfrak{F}$.
The following result shows a connection between an $\mathfrak{F}$-projector of a group and its $\mathfrak{F}$-residual, when the latter is abelian.

\begin{thm}[Theorem IV.5.18 of \cite{Doerk}]\label{residual}
Suppose that  $\mathfrak{F}$ is a saturated formation. Let $G$ be a solvable group and $G^{\mathfrak F}$ its $\mathfrak{F}$-residual. If $G^{\mathfrak F}$ is abelian, then $G^{\mathfrak F}$ is complemented in $G$, and its complements are the $\mathfrak{F}$-projectors of $G$.
\end{thm}

\begin{remark}\label{caja}
As a consequence of the previous result, let $\mathfrak{F}$ be a saturated formation, $G$ a solvable group and  $K=G^\mathfrak{F}$. If  $L \nor G$ is such that $L\subseteq K$ and $K/L$ is abelian, then $$G=KH  \text{ and } K \cap LH=L\,, $$ 
\noindent 
where $H$ is an $\mathfrak{F}$-projector of $G$. 
Furthermore, under these conditions, using Proposition \ref{basics}(c), any complement of $K/L$ in $G/L$ is of the form $H^gL/L$ for some $g \in G$.
\end{remark}

\smallskip

We finish this section with a result that can be found as Satz 2.1 of \cite{H} or Theorem IV.5.4 of \cite{Doerk}.

\begin{lemma}\label{hup} Suppose that $\mathfrak{F}$ is a saturated formation. Let $G$ be a finite solvable group and $H$ an $\mathfrak{F}$-projector of $G$. If $N, M$ are normal subgroups of  $G$, then 
$$MN\cap H=(M\cap H)(N\cap H)\,.$$
\end{lemma}

\section{Extensions of irreducible characters}

In this section, we present results on extensions of irreducible characters that behave well with certain subgroups and saturated formations. We first recall two results on character correspondences.

\begin{lemma}[Corollary 4.3 of \cite{Ispi}] \label{induccionbiyeccion}
Let $N \trianglelefteq G$ and $K \subseteq G$ with $NK=G$ and $N \cap K=M$. Let $\varphi \in \irr M$ be invariant in $K$, and assume $\theta=\varphi^N$ is irreducible. Then induction defines a bijection $\irr {K \mid \varphi} \rightarrow \irr{G \mid \theta}$.
\end{lemma}

\begin{lemma}[Corollary 4.2 of \cite{Ispi}]\label{restriccionbiyeccion}
Let $N \trianglelefteq G$ and $K \subseteq G$ with $NK=G$ and $N \cap K = M$. Let $\theta\in\irr N$ be invariant in $G$, and assume $\varphi=\theta_M $ is irreducible. Then restriction defines a bijection $\irr{G \mid \theta} \rightarrow \irr{K \mid \varphi}$.
\end{lemma}

Given an $\mathfrak{F}$-projector $H$ of a solvable group $G$, we adopt the usual notation $\irrh N$ for the $H$-invariant irreducible characters of an $H$-invariant subgroup $N$ of $G$.  The following theorem, due to Navarro and appearing in \cite{N}, is key to our objectives. Note that the subgroup denoted by $H$ in the cited reference corresponds to $LH$ in our notation.

\begin{thm}[Theorem 3.5 of \cite{N}]\label{3.5 original}

Let $\mathfrak{F}$ be a saturated formation. Suppose $G$ is a finite solvable group. Assume that $K/L$ is abelian, where $K$ and $L$ are normal in $G$. Let $H$ be an $\mathfrak{F}$-projector of $G$, and assume that $KH \trianglelefteq G$ and $K \cap \norm G H L = L$. We have the following.
\begin{itemize}
    \item[(a)] If $\theta \in \irrh K$, then there exists a unique $\varphi \in \irrh L$ under $\theta$.
    \item[(b)] If $\varphi \in \irrh L$, then there exists a unique $\theta \in \irrh K$ over $\varphi$.
    \end{itemize}
    Now fix $\varphi$ and $\theta$ as before.
    \begin{itemize}
    \item[(c)] We have that $\theta$ extends to $KH$ if and only if $\varphi$ extends to $LH$.
\end{itemize}

\end{thm}

Under the hypotheses of Theorem \ref{3.5 original}, we have that $G=K \norm G H$ by Proposition \ref{basics}(d). Notice that Theorem \ref{3.5 original}(c) gives us information on irreducible character extensions from $K$ to $KH$. We wish to have control on irreducible character extensions to $G$, but we do not always have that $KH=G$. In particular, by Proposition \ref{basics}(e), this happens for saturated formations such that $\mathfrak{N} \subseteq \mathfrak{F}$. In this case, still under the hypotheses of Theorem \ref{3.5 original}, we have that $K/L$ is abelian, $KH=G$ and $K\cap LH=L$.

Whenever $K$ and $L$ are normal subgroups of $G$ such that $K/L$ is abelian, $KH=G$ and $K \cap LH=L$ where $H$ is an $\mathfrak{F}$-projector of $G$, we will say that $(G,K,L)$ satisfies the \textit{Navarro condition} with respect to $\mathfrak{F}$. Note that, for instance, if $\mathfrak{N} \subseteq \mathfrak{F}$,  $K=G^{\mathfrak{F}}$ and $L=K'$, then $(G,K,L)$ satisfies the Navarro condition with respect to $\mathfrak{F}$ by Remark \ref{caja}.
Now, Theorem \ref{3.5 original} can be formulated as follows.

\begin{thm} \label{key}
 Let $\mathfrak{F}$ be a saturated formation with $\mathfrak{N}\sbs \mathfrak{F}$. Suppose $(G,K,L)$ satisfies the Navarro condition with respect to $\mathfrak{F}$. We have the following.
\begin{itemize}
\item[(a)]
If $\theta \in \irrh K$, then there exists a unique $\varphi$  in $\irrh L$ under $\theta$.
\item[(b)] If $\varphi \in \irrh L$, then there exists a unique $\theta$ in $\irrh K$ over $\varphi$.
\end{itemize}
Now fix $\varphi$ and $\theta$ as before.
\begin{itemize}
\item[(c)] We have that  $\theta$ extends to $G$ if and only if $\varphi$ extends to $LH$.
\end{itemize}
\end{thm}

We are interested in studying whether there is any relationship between the two sets of $H$-invariant character extensions that intervene in the previous theorem.
The result below was generously communicated to us by G. Navarro and P. H. Tiep (see Theorem 5.2 of \cite{NT}). We are thankful to them for allowing us to present it here.

\begin{thm}[Navarro, Tiep] \label{malle&navarro&tiep}
Suppose $G$ is a finite solvable group, and let $K,L\nor G $ such that $K/L$ is abelian. Suppose  that $H/L$ is a self-normalizing nilpotent subgroup of $G/L$, and  suppose that $G= KH$ and $K\cap H=L$. Let $\theta \in \irrh L$ and let $\varphi \in \irrh K$ lying over $\theta$. Then the following hold.
\begin{itemize}
\item[(a)] If $\chi\in \irr G$ extends $\theta$, then there is an irreducible constituent of $\chi_H$ extending $\varphi$. 
\item[(b)] If $\nu\in \irr H$ extends $\varphi$, then there is an irreducible constituent of $\nu^G$ extending  $\theta$.
\end{itemize}
\end{thm}

In the proof of the following lemma, we will use Isaacs' notation from \cite{brown}.

\begin{lemma}\label{fully}
Let $\mathfrak{F}$ be a saturated formation with $\mathfrak{N} \subseteq \mathfrak{F}$. Suppose that $K$ and $L$ are normal subgroups of a solvable group $G$ such that $K/L$ is a chief factor of $G$, $K \cap LH = L$, and $G = KH$, where $H$ is an $\mathfrak{F}$-projector of $G$.  Moreover, assume that $G$ has odd order whenever $\mathfrak{N} \neq \mathfrak{F}$. Let $\varphi \in \irrh {L}$ be fully ramified with respect to $K/L$, and let $\theta \in \mathrm{Irr}(K)$ be the unique $H$-invariant irreducible character lying over $\varphi$. 
\begin{itemize}
    \item[(a)] If $\eta$ is an irreducible extension of $\varphi$ to $LH$, then there exists an extension $\chi \in \mathrm{Irr}(G)$ of $\theta$ lying over $\eta$.
    \item[(b)] If $\chi$ is an irreducible extension of $\theta$ to $G$, then there exists an extension $\eta \in \mathrm{Irr}(LH)$ of $\varphi$ lying under $\chi$.
\end{itemize}
\end{lemma}

\begin{proof}
 If $\mathfrak{F} =\mathfrak{N}$, then an $\mathfrak{F}$-projector is a Carter subgroup $C$ of $G$, and the result follows from Theorem~\ref{malle&navarro&tiep}, considering $CL/L$ as the self-normalizing nilpotent subgroup of $G/L$. We may thus assume that $\mathfrak{N} \subseteq \mathfrak{F}$ and that $G$ has odd order.

We have that $(G,K,L,\theta,\varphi)$ is a character five, as described in Section 3 of \cite{brown}. By applying Theorem \ref{residual} to $G/L$, we deduce that any complement of $K/L$ in $G$ is conjugate to $LH$. Let $\psi$ be the canonical character of $G$ defined in \cite{brown}. We know, by Corollary 5.9 of \cite{brown}, that $\psi$ contains  the trivial character $1_G$ as an irreducible constituent. Now,  let  $\eta$ be  an irreducible extension of $\varphi$ to $LH$; then, by Theorem 9.1~(c) of \cite{brown},   there is a unique $\chi \in \irr G$ over $\theta$ such that $$\chi_{LH}= \psi_{LH} \ \eta,$$ and hence $\chi$ lies over $\eta$. By comparing degrees (taking into account that $\psi(1)^2=|K/L|$), we conclude that $\chi_K=\theta$. Similarly, one can show that if $\chi$ is an irreducible extension of $\theta$ to $G$, then $\chi_{LH}$ has a constituent $\eta \in \irr{LH\mid \varphi}$, which is an extension of $\varphi$.
\end{proof}

We prove that the $H$-invariant extensions of irreducible characters from $K$ to $KH$, as given in Theorem~\ref{key}, lie over the $H$-invariant extensions from $L$ to $LH$, provided that $G$ has odd order or that $\mathfrak{F}$ is the formation of nilpotent groups.

\medskip

Note that the hypothesis that $|G|$ is odd is necessary in Lemma \ref{fully}, and therefore in the result below, as shown by considering the group $G=\texttt{SmallGroup(48,28)}$, the saturated formation $\mathfrak{U}$ of supersolvable groups, $K=G^\mathfrak{U}$ the $\mathfrak{U}$-residual of $G$ and $L=K'$ the derived subgroup of $K$. Let as usual $H$ be a $\mathfrak{U}$-projector of $G$. In this case, the only nonlinear character $\theta$ of $K \cong Q_8$ lies over the nontrivial character $\varphi$ of $L \cong C_2$ and they extend to $G$ and $LH$ respectively (thus are both $H$-invariant). However, no extension of $\theta$ to $G$ lies over an extension of $\varphi$ to $LH$.

\begin{thm}\label{cor1}
Let $\mathfrak{F}$ be a saturated formation with $\mathfrak{N} \subseteq \mathfrak{F}$, let $G$ be a solvable group, and let $H$ be an $\mathfrak{F}$-projector of $G$. Suppose that the triple $(G, K, L)$ satisfies the Navarro condition with respect to $\mathfrak{F}$. Further assume that $G$ has odd order whenever $\mathfrak{N} \neq \mathfrak{F}$.  Let $\varphi\in\irrh L$, and let $\theta\in\irrh {K}$ be the unique  $H$-invariant irreducible  character  over $\varphi$. 
\begin{itemize}
\item[(a)] If $\eta $ is an irreducible extension of $\varphi$ to $LH$, then there exists an extension  $\chi \in \irr G$ of $\theta$ lying over $\eta$. 
\item[(b)] If  $\chi$ is  an irreducible extension of $\theta$, then there exists an  extension $\eta \in \irr{LH}$ of $\varphi$ lying under $\chi$.
\end{itemize}
\end{thm}

\begin{proof}
Working by induction on $|K : L|$ and $|G : K|$, we can assume that $K/L$ is a chief factor of $G$. Otherwise, let $K/M$ be a chief factor of $G$, and $\mu \in \text{Irr}_H(M)$ be the unique $H$-invariant irreducible character under  $\theta$ and over $\varphi$ (this follows from Theorem \ref{key}). 

By induction, we have that if $\eta$ is an irreducible extension of $\varphi$ to $LH$, then there exists  an irreducible  extension $\gamma$  of the character $\mu$ to $HM$ lying over $\eta$. 
Then, again by induction, we have that there is  an irreducible extension $\chi$ of the character $\theta$ to $G$ lying over $\gamma$. Hence, $\chi$ lies over $\eta$ and (a) is proved. 

Conversely, if $\chi$  is an  irreducible extension of $\theta$ to $G$, then, by induction, there exists  an irreducible extension $\gamma$  of the character $\mu$ to $HM$ lying under $\chi$. Then, again by induction, we have that there  is an extension  $\eta$ of the character  $\varphi$ to $LH$ lying under $\eta$. Therefore, $\eta$ lies under $\chi$ and we have (b).

We may also assume that $\varphi$ is $G$-invariant. Otherwise, since the inertia subgroup $ I_G(\varphi) = LH$ is normal in $G$ and $K/L$ is a chief factor of $G$, we have that $I_K(\varphi) = L$. Then, by Lemma  \ref{induccionbiyeccion}, induction defines a bijection  from $\text{Irr}(LH \mid \varphi) $ onto $\text{Irr}(G \mid \theta)$, and (a) and (b) follow.

Now, by  the ``going down" theorem, we have that either $\theta$ is fully ramified with respect to $K/L$ or $\theta$ is an irreducible extension of $\varphi$ to $K$. In the latter case, we have that the restriction defines a bijection $\irr{G \mid \theta} \rightarrow \irr{K \mid \varphi}$ by Lemma \ref{restriccionbiyeccion}, and the result follows.

So, we may assume that $\theta$ is fully ramified with respect to $K/L$.  Therefore the result holds by Lemma \ref{fully}.
\end{proof}
\section{$\mathfrak{F}'$-characters}

In this section, we present the $\mathfrak{F}'$-characters defined in~\cite{N}, focusing on the specific case where the saturated formation $\mathfrak{F}$ contains $\mathfrak{N}$.

Recall that saturated formations containing the nilpotent groups have the property that their projectors in solvable groups are self-normalizing. Thus, in this setting, if $G$ is a solvable group and $H$ is an $\mathfrak{F}$-projector of $G$, then any normal subgroup $N \unlhd G$ such that $NH \nor G$ must satisfy $NH = G$ (see Proposition~\ref{basics}).

The construction in \cite{N} is based on Navarro's definition of $G^{\mathfrak{F}_n}$, the unique smallest normal subgroup of $G$ such that $G^{\mathfrak{F}_n}H \unlhd G$. From the discussion above, in our context we have $G^{\mathfrak{F}_n} = G^{\mathfrak{F}}$, the $\mathfrak{F}$-residual of $G$. Since $G$ is solvable and $\mathfrak{N} \subseteq \mathfrak{F}$, we have that $G^\mathfrak{F}<G$ when $G>1$.

If $G \in \mathfrak{F}$, then the $\mathfrak{F}'$-characters are defined to be precisely the linear characters of $G$. Assume now that $G \notin \mathfrak{F}$. We introduce the notation that will be used throughout the remainder of the paper.

Define $K_0 = G^{\mathfrak{F}}$ and $L_0 = K_0'$. Now, let  $K_i = (L_{i-1}H)^{\mathfrak{F}}$, and $L_i = K_i'$ for $i \geq 1$. We have that $K_iH=L_{i-1}H$ and $K_i\leq L_{i-1}$ for all $i\geq 1$, and that $(K_iH, K_i, L_i)$ satisfies the  Navarro condition with respect to $\mathfrak{F}$ for all $i\geq 0$. There exists a non-negative integer $m$ such that
$$
G = K_0H > L_0H>\cdots > K_{m-1}H> L_{m-1}H = H.
$$
At this point, we reach $K_m = (L_{m-1}H)^{\mathfrak{F}} = H^{\mathfrak{F}} = 1$.

Therefore, we obtain the following $H$-invariant subnormal series
$$
G \rhd K_0 \rhd L_0 \rhd K_1 \rhd L_1  \rhd \cdots\rhd K_{m-1} \rhd L_{m-1} \rhd K_m=1.
$$

Navarro constructs the $\mathfrak{F}'$-characters of $G$ through an ascending process that uses the tuples $(L_{i-1}\norm G H, K_i,L_i,L_iH,L_i\norm G H)$. In the case $\mathfrak{N} \subseteq \mathfrak{F}$, $H=\norm G H$ and the last two entries of the tuples coincide. Let us describe the process of construction of the $\mathfrak{F}'$-characters in this case. We denote the set of these characters by $\irrf G$.

Let $\Delta$ be a subset of $\irr{N}$ for some normal subgroup $N$ of $G$. We define
$$\irr{G \mid \Delta} = \bigcup_{\lambda \in \Delta} \irr{G \mid \lambda},$$
where $\irr{G \mid \lambda}$ denotes the set of irreducible characters of $G$ lying over the irreducible character $\lambda$.

As indicated earlier, we begin with $ K_{m}H=H$, so $\irrf{H} = \text{Lin}(H)$, the set of linear characters of $H$.

Define
$$
\Delta_{m-1} = \{ \chi_{L_{m-1}} \mid \chi \in \irrf{H} \} \subseteq \irr{L_{m-1}},
$$
and construct
$$
\irrf{L_{m-2}H} = \{ \chi \in \irr{L_{m-2}H \mid \Delta_{m-1}} \mid \chi_{K_{m-1}} \in \irr{K_{m-1}} \}.
$$
Proceeding recursively, define
$$
\Delta_{m-2} = \{ \chi_{L_{m-2}} \mid \chi \in \irrf{L_{m-2}H} \},
$$
and
$$
\irrf{L_{m-3}H} = \{ \chi \in \irr{L_{m-3}H \mid \Delta_{m-2}} \mid \chi_{K_{m-2}} \in \irr{K_{m-2}} \}.
$$
This process continues until we reach
$$
\Delta_0 = \{ \chi_{L_0} \mid \chi \in \irrf{L_0H} \},
$$
and finally
$$
\irrf{G} = \irrf{K_0H} = \{ \chi \in \irr{G \mid \Delta_0} \mid \chi_{K_0} \in \irr{K_0} \}.
$$

We introduce the notation $(N, \varphi) \triangleleft (M, \theta)$ to indicate that $N$ is a normal subgroup of $M$, and that the irreducible character $\theta$ of $M$ lies over the irreducible character $\varphi$ of $N$.

Note that since $K_i \leq L_{i-1}$ for all $1 \leq i \leq m$, if $\chi \in \irrf G$, there exist characters $\theta_i \in \irr{K_i}$ and $\varphi_i \in \irr{L_i}$, extending to $K_iH$ and $L_iH$, respectively, such that 
$$
(1, 1) =(K_{m}, \theta_{m})\triangleleft (L_{m-1}, \varphi_{m-1}) \triangleleft \cdots\triangleleft (K_{1}, \theta_{1}) \triangleleft (L_0, \varphi_0) \triangleleft (K_0, \theta_0) \triangleleft (G, \chi),
$$
with $(\varphi_i)_{K_{i+1}} = \theta_{i+1}$, for $0 \leq i\leq m-1$.

\smallskip

This allows us to consider a characterization of the $\mathfrak{F}'$-characters by means of a descending series, as follows.

\begin{thm} \label{caracterización descendente}
Suppose that $\mathfrak{F}$ is a saturated formation with $\mathfrak{N} \subseteq \mathfrak{F}$. Let $G$ be a solvable group, and let $H$ be an $\mathfrak{F}$-projector of $G$. Denote $K_0=G^\mathfrak{F}$ and $L_0=K_0'$. For $i \geq 1$, set $K_i=(L_{i-1}H)^\mathfrak{F}$ and $L_i=K_i'$. Then $\chi \in \irrf G$ if and only if the following conditions hold.
\begin{description}
\item[(a)] The character $\chi_{K_0}$ is irreducible. We denote $\theta_0 = \chi_{K_0}$.
\item[(b)] There exist characters $\varphi_i \in \irr{L_i}$ and $\theta_i \in \irr{K_i\mid \varphi_i}$, for $0 \leq i \leq m - 1$, extending to $L_iH$ and $K_iH$, respectively, such that
$$
(1, 1) = (K_{m}, \theta_{m}) \triangleleft (L_{m-1}, \varphi_{m-1}) \triangleleft \cdots \triangleleft (K_{1}, \theta_{1}) \triangleleft (L_0, \varphi_0) \triangleleft (K_0, \theta_0) \triangleleft (K_0H, \chi)=(G, \chi),
$$
and $(\varphi_i)_{K_{i+1}} = \theta_{i+1}$, for $0 \leq i\leq m-1$.
\end{description}
\end{thm}
\begin{proof}

If $\chi \in \irrf G$, then conditions (a) and (b) follow from the comments preceding this theorem.

We now prove the converse by induction on $|G:H|$.
Assume that (a) and (b) hold for an irreducible character $\chi$ of $G$.  

If $G$ is an $\mathfrak{F}$-group, \textit{i.e.}, if $G=H$, then $K_0 = 1$, and by (a), it follows that $\chi$ is a linear character. Hence, $\chi$ is an $\mathfrak{F}'$-character of $G$.

We may therefore assume that $G \notin \mathfrak{F}$; from this,  it follows that  $L_0H$ is a proper subgroup  of  $G$.   Let $\eta$ be an extension of $\theta_1$ to $K_1H = L_0H$ such that $\eta_{L_0} = \varphi_0$. Then we have
$$
(1, 1) = (K_m, \theta_m) \triangleleft (L_{m-1}, \varphi_{m-1}) \triangleleft \cdots \triangleleft (K_1, \theta_1) \triangleleft (L_0, \varphi_0) \triangleleft (L_0H, \eta),
$$

\noindent
and therefore, conditions (a) and (b) hold for $\eta$. 

By the inductive hypothesis, we have $\eta \in \irrf{L_0H}$. Hence,  $\varphi_0 \in \Delta_0= \{ \psi_{L_0} \mid \psi \in \irrf{L_1H} \}$. Now, since  $\chi \in \irr{G \mid \Delta}$, it follows that $\chi \in \irrf G$, which completes the proof.
\end{proof}

\begin{remark}
As we mentioned, if a solvable group $G$ lies in $\mathfrak{F}$, then its $\mathfrak{F}'$-characters coincide with its linear characters. The converse is false; for example, the symmetric group $\mathfrak{S}_4$ is not supersolvable but $\text{Irr}_\mathfrak{U}(\mathfrak{S}_4)=\text{Lin}(\mathfrak{S}_4)$. However, observe that linear characters of $G$ automatically satisfy the conditions of Theorem \ref{caracterización descendente} and thus are always included in the $\mathfrak{F}'$-characters of $G$.
\end{remark}

Next, we provide information about the number of $\mathfrak{F}'$-characters of a solvable group, in the case where $\mathfrak{F}$ contains the formation of nilpotent groups.

\begin{proposition}
Let $\mathfrak{F}$ be a saturated formation with $\mathfrak{N} \subseteq \mathfrak{F}$. Let $G$ be a solvable group, and let $H$ be an $\mathfrak{F}$-projector of $G$. Then  
$$
|\irrf G| = |\irr{H/H'}|\,.
$$  
That is, the number of $\mathfrak{F}'$-characters of $G$ coincides with the number of linear characters of $H$.  
\end{proposition}
\begin{proof}
This follows from Theorem A of \cite{N}.
\end{proof}

\section{Strong pair series}

In this section, we aim to present an alternative characterization of the $\mathfrak{F}'$-characters. This characterization is analogous to the one given by Isaacs in \cite{Isaacs} for the formation $\mathfrak{N}$ of nilpotent groups, but it extends to all saturated formations containing $\mathfrak{N}$. 

 Suppose that $\mathfrak{F}$ is a saturated formation with $\mathfrak{N} \subseteq \mathfrak{F}$. Let $G$ be a solvable group and $H$ an $\mathfrak{F}$-projector  of $G$.
Suppose that  $\{S_i \}_{0 \leq i \leq r}$ is an $H$-composition series for $G$. This is an $H$-invariant, subnormal series such that $S_{i+1}/S_i$ is $H$-simple (and thus abelian) for $0 \leq i \leq r-1$. The \textit{$H$-composition length} $r$ is an invariant of $G$.

We say that $\{(S_i, \theta_i) \}_{0 \leq i \leq r}$ is an \textit{$H$-pair series} for $G$ that is associated with the $H$-composition series $\{S_i \}_{0 \leq i \leq r}$, if the characters $\theta_i \in\irr{S_i}$ are $H$-invariant   and $\theta_i$ lies under $\theta_{i+1}$ for $0 \leq i < r$. Usually, we will write the $H$-pair series as  $$(1,1) =(S_0, \theta_0)  \triangleleft (S_1, \theta_1) \triangleleft \cdots \triangleleft (S_r, \theta_r)$$ to simplify the understanding of the chain of pairs being used.

We say that an $H$-pair series $\{(S_i, \theta_i) \}_{0 \leq i \leq r}$ for $G$ is \textit{strong} if each of the characters $\theta_i$ extends irreducibly to $S_iH$. An irreducible character $\chi$ of $G$ is said to be an \textit{$\mathfrak{F}$-head character} of $G$ if $(G,\chi)$ appears as $(S_r,\theta_r)$ in a strong  $H$-pair series.

\begin{remark}  \label{remark 3}
It is worth noting that these irreducible character extensions define, in turn, $\mathfrak{F}$-head characters of $S_iH$. Indeed, fixing $ i \in \{0, \ldots, r\}$, we denote by $\eta_i$ an irreducible extension of $\theta_i$ to $S_iH$ and we define $\{T_j\}_{0 \leq j \leq m}$  such that $T_0=S_i$, $T_m=S_iH$ and $T_{j+1}/T_j$ is an $H$-composition factor for all $j$. Then one can easily verify that
$$(1,1)=(S_0,\theta_0) \triangleleft \cdots \triangleleft (S_i,\theta_i) = (T_0,\eta_{i_{\mid T_0}}) \triangleleft \cdots \triangleleft (T_j,\eta_{i_{\mid T_j}})\triangleleft \cdots \triangleleft (S_iH,\eta_i)=(T_m,\eta_i) $$
is a strong $H$-pair series for $S_iH$.
\end{remark}

\begin{lemma}\label{MH=UH} Let $\mathfrak{F}$ be a saturated formation with $\mathfrak{N} \subseteq \mathfrak{F}$. Let $G$ be a solvable group, and let $H$ be an $\mathfrak{F}$-projector of $G$.   Suppose that $M/U$ is an $H$-composition factor of $G$. Let $\alpha$ be an irreducible character of $M$ lying over the irreducible character $\theta$ of $U$. If $\alpha$ and $\theta$ extend irreducibly to $MH$ and $UH$ respectively, then $\theta$ is the unique $H$-invariant irreducible constituent of $U$ lying under $\alpha$. Moreover, if $MH=UH$, then $\alpha_U = \theta$.\end{lemma} 
 \begin{proof}
We have that $M \cap UH$ is a normal subgroup of $MH$ that contains $U$ and is contained in $M$. Since $M/U$ is a chief factor of $MH$, it follows that $M \cap UH$ is either $U$ or $M$. 

If $M \cap UH = U$, then $(MH, M, U)$ satisfies the Navarro condition with respect to $\mathfrak{F}$, and by Theorem~\ref{key}, we have that $\theta$ is the unique $H$-invariant character of $U$ lying under $\alpha$.

In the case where $M \cap UH = M$, we have $MH = UH$. If $\tilde{\theta}$ is an extension of $\theta$ to $MH$, then $\tilde{\theta}_M$ is irreducible. Since $M/U$ is abelian, by Gallagher's theorem, we have $\alpha = \lambda  \tilde{\theta}_M$ for some $\lambda \in \text{Lin}(M/U)$. Therefore, $\alpha_U = \theta$, which completes the proof.
\end{proof} 

\begin{lemma} \label{5.2}
    Let $\mathfrak{F}$ be a saturated formation with $\mathfrak{N} \subseteq \mathfrak{F}$. Let $G$ be a solvable group, and let $H$ be an $\mathfrak{F}$-projector of $G$. Suppose that $M/V$ and $M/U$ are two distinct $H$-composition factors of $G$, and let $D = U \cap V$. If 
$$(D, \gamma) \triangleleft (U, \theta) \triangleleft (M, \alpha)\,,$$  
with $\gamma$, $\theta$, and $\alpha$ extending irreducibly to $DH$, $UH$, and $MH$, respectively, then there exists an $H$-invariant irreducible character $\beta$ of $V$, which extends irreducibly to $VH$ and satisfies  
$$(D, \gamma) \triangleleft (V, \beta) \triangleleft (M, \alpha)\,.$$
\end{lemma}

\begin{proof} We begin with some preliminary observations.

Since $M/U$ and $M/V$ are abelian chief factors of $MH$, we have $M' \leq U \cap V = D$, and thus $M/D$ is abelian. As $U$ and $V$ are distinct, we have that $UV = M$.

The subgroup $M \cap UH$ is normal in $M$, $H$-invariant, and contains $U$. As $M/U$ is a chief factor of $MH$, either $M \cap UH = M$ or $M \cap UH = U$. Similarly, since $M/U \cong V/D$ and $M/V \cong U/D$ are chief factors of $MH$, we have
$$DH \cap U \in \{D, U\}, \quad M \cap VH \in \{M, V\}, \quad \text{and} \quad DH \cap V \in \{D, V\}.$$

Now assume $UH < MH$. Then $UH \cap M = U$ and $DH < VH$. To verify the latter, suppose $DH = VH$. Then
$$V = DH \cap V = D(H \cap V),$$
and therefore
$$MH = UVH = UD(H \cap V)H = UH,$$
a contradiction. Similarly, if $VH < MH$, then $VH \cap M = V$ and $DH < UH$.

Conversely, if $UH = MH$, then $VH = DH$. Suppose, for the sake of contradiction, that $DH < VH$. Then
$$D = DH \cap V = D(H \cap V),$$
which implies $V \cap H \leq D$. By Lemma~\ref{hup}, we have
$$M \cap H = (U \cap H)(V \cap H) \leq U.$$
Hence,
$$M = M \cap MH = M \cap UH = U(M \cap H) = U,$$
a contradiction. Similarly, if $VH = MH$, then $DH = UH$.

We analyze the four possible situations:

\medskip
\textbf{(1) $MH = UH = VH$.}

\medskip

Since $MH = UH$ by our previous argument, we conclude that $VH = DH$, and therefore
$$DH = UH = VH = MH.$$
By Lemma~\ref{MH=UH}, $\alpha_U = \theta$ and $\theta_D = \gamma$, so $\alpha$ extends $\gamma$, and we take $\beta = \alpha_V$.

\medskip
\textbf{(2)} $VH < MH$ and $MH = UH$.
\medskip

We have $D = U \cap DH$, $V = M \cap VH$, and $VH = DH$. The triple $(MH, M, V)$ satisfies the Navarro condition with respect to $\mathfrak{F}$. By Theorem~\ref{key}, there exists a unique $\beta \in \irrh V$ under $\alpha$.

Since $MH = UH$, by Lemma~\ref{MH=UH}, we have $\alpha_U = \theta$. As $U/D$ is a chief factor of $UH = MH$, we consider two possibilities (by Exercise 6.12 of \cite{Is}):

\medskip
\textbf{(2.1)} $\gamma$ is $U$-invariant (and thus $M$-invariant). Then the only irreducible constituent of $\alpha_D = \theta_D$ is $\gamma$. Thus, any irreducible character $\beta$ of $V$ lying above $\alpha$ must lie above $\gamma$, and therefore $\beta$ lies over $\gamma$.

\medskip
\textbf{(2.2)} $\theta = \gamma^U$. Here $\gamma$ is $V$-invariant since $DH \cap V = V = D(H \cap V).$
The hypotheses of Lemma~\ref{induccionbiyeccion} are satisfied, and induction defines a bijection
$$\irr{V \mid \gamma} \rightarrow \irr{M \mid \theta}.$$
Let $\beta' \in \irr{V \mid \gamma}$ be such that $\alpha = (\beta')^M$. For any $h \in H$, we have
$$(\beta')^M = \alpha = \alpha^h = ((\beta')^M)^h = ((\beta')^h)^M,$$
with $(\beta')^h \in \irr{V \mid \alpha}$, since $\alpha$ is $H$-invariant. Hence, $(\beta')^h = \beta'$ for all $h \in H$, so $\beta'$ is an $H$-invariant constituent of $\alpha_V$. By the uniqueness of $\beta$ as an $H$-invariant constituent of $\alpha_V$, it follows that $\beta = \beta'$.

\medskip
\textbf{(3)} $UH < MH$ and $MH = VH$.
\medskip

We have $DH = UH$ and $U = M \cap UH$. By Lemma~\ref{MH=UH}, it follows that $\theta_D = \gamma$. Observe that $(VH, V, D)$ satisfies the Navarro condition with respect to $\mathfrak{F}$. Then by Theorem~\ref{key}, there exists a unique character $\beta \in \irrh{V \mid \gamma}$. We want to show that $\beta$ is a constituent of $\alpha_V$.

Since $M/U$ is a chief factor of $MH$, by the ``going down'' theorem (Theorem 6.18 of \cite{Is}), we have two possibilities:

\medskip

\textbf{(3.1)} $\theta$ is $M$-invariant. Then since $\theta_D = \gamma$ and $\theta$ is $M$-invariant, we can apply Lemma~\ref{restriccionbiyeccion}, which states that restriction defines a bijection from $\irr{M \mid \theta}$ to $\irr{V \mid \gamma}$. Hence, $\alpha_V \in \irrh{V \mid \gamma}$, and thus $\beta = \alpha_V$. 

\medskip
\textbf{(3.2)} $\alpha = \theta^M$. In this case, since $\gamma^V = (\theta_D)^V = (\theta^M)_V = \alpha_V$, and $\beta \in \irrh{V \mid \gamma}$, we deduce that $\beta$ lies under $\alpha$.

\medskip
\textbf{(4)} $UH < MH$ and $VH < MH$.
\medskip

It follows that $UH \cap M = U$, $VH \cap M = V$, $DH < UH$, and $DH < VH$. Since both $(MH, M, V)$ and $(VH, V, D)$ satisfy the Navarro condition with respect to $\mathfrak{F}$, the theorem follows by applying Theorem~\ref{key} twice.
\end{proof}

The proof of the theorem below follows that of Theorem 5.5 of \cite{Isaacs}, differing mainly in the final step.

\begin{thm}\label{eslabon}
   Let $\mathfrak{F}$ be a saturated formation with $\mathfrak{N} \subseteq \mathfrak{F}$.  Let $G$ be a solvable group, and let $H$ be an $\mathfrak{F}$-projector of $G$. Suppose $\{T_i\}_{0 \leq i \leq r}$ is an arbitrary $H$-composition series for $G$. Let $\chi \in \irr G$ be an $\mathfrak{F}$-head character of $G$. Then the series $\{T_i\}_{0 \leq i \leq r}$ is associated with some unique strong $H$-pair series $\{(T_i, \varphi_i)\}_{0 \leq i \leq r}$ such that $\varphi_r = \chi$.
\end{thm}

\begin{proof}
Since  $\chi$ is an $\mathfrak{F}$-head character of $G$, by definition, there exists a strong $H$-pair series 
$$(1,1) =(S_0, \theta_0)  \triangleleft (S_1, \theta_1) \triangleleft \cdots \triangleleft (S_r, \theta_r)=(G, \chi)\,.$$  We want to show the existence of characters $\varphi_i \in \irr {T_i}$ satisfying that $$(1,1) =(T_0, \varphi_0)  \triangleleft (T_1, \varphi_1) \triangleleft \cdots \triangleleft (T_r, \varphi_r)=(G, \chi)\,$$ is a strong $H$-pair series. 

Now, note that $S_r = G = T_r$. Therefore, there exists a unique smallest non-negative integer $m$ such that $S_i = T_i$ for all $i$ satisfying $m \leq i \leq r$. If $m = 0$, then $T_i = S_i$ for all $i \in \{0,\hdots,r\}$, and we can set $\varphi_i = \theta_i$, concluding the proof. Thus, we assume $m > 0$ and proceed by downward induction on $m$.

 We have $S_m = T_m$, but $S_{m-1} \neq T_{m-1}$. Since $S_0 = 1 = T_0$, it follows that $m-1 > 0$, and hence $m \geq 2$. Define $M = S_m = T_m$, $U = S_{m-1}$, and $V = T_{m-1}$. Write $D=U \cap V$. Clearly, $U$ and $V$ are nontrivial and distinct. Furthermore, $M/U$ and $M/V$ are $H$-composition factors of $G$, so $U$ and $V$ are maximal among $H$-invariant normal subgroups of $M$. As $U \neq V$, it follows that $UV = M$. Also, $U/D$ is $H$-isomorphic to  $M/V$, thus $U/D$ is an $H$-composition  factor of $G$, and similarly, $V/D$ is an $H$-composition factor of $G$.

The $H$-composition length of $U = S_{m-1}$ is $m-1$, implying that the $H$-composition length of $D$ is $m-2$. Thus, we can choose an $H$-composition series 
$$1 = D_0 \lhd D_1 \lhd \cdots \lhd D_{m-2} = D$$
for $D$. Appending the subgroups $S_i$ for $m-1 \leq i \leq r$ to this series yields a new $H$-composition series for $G$. Denote this series by $\{N_i\}_{0 \leq i \leq r}$, where $N_i = D_i$ for $0 \leq i \leq m-2$, and $N_i = S_i$ for $m-1 \leq i \leq r$. In particular, $N_{m-1} = S_{m-1} = U$ and $N_{m-2} = D_{m-2} = D$.

Using the inductive hypothesis with the series $\{N_i\}_{0 \leq i \leq r}$ in place of $\{T_i\}_{0 \leq i \leq r}$ and $m-1$ in place of $m$, we conclude that $\{N_i\}_{0\leq i \leq r}$ is associated with a strong $H$-pair series with $\mathfrak{F}$-head character $\chi$. Moreover, for $i \geq m$, we have $T_i = S_i = N_i$. Hence, we may replace $\{S_i\}_{0 \leq i \leq r}$ with $\{N_i\}_{0 \leq i \leq r}$ without loss of generality, and we can assume $N_i = S_i$ for all $i \in \{0,\hdots,r\}$. In particular, $D = D_{m-2} = N_{m-2} = S_{m-2}$.

Next, construct a new $H$-composition series for $G$ by modifying $\{S_i\}_{0 \leq i \leq r}$, replacing $S_{m-1} = U$ with $V$. Note that $S_{m-2} = D \lhd V \lhd M = S_m$. And since $V/D$ is $H$-isomorphic to $M/U = S_m/S_{m-1}$ and $M/V$ is $H$-isomorphic to $T_m/T_{m-1}$, the series is an $H$-composition series. We denote this new  series by $\{X_i\}_{0 \leq i \leq r}$, where $X_i = S_i$ for $i \neq m-1$ and $X_{m-1} = V$. Now, since $\{S_i\}_{0 \leq i \leq r}$ is associated to a strong $H$-pair series with $\mathfrak{F}$-head character $\chi$, using Lemma \ref{5.2}, that is also the case for $\{X_i\}_{0 \leq i \leq r}$.

Finally, note that $T_i = S_i = X_i$ for $i \geq m$ and $T_{m-1} = V = X_{m-1}$, so the series $\{X_i\}_{0 \leq i \leq r}$ and $\{T_i\}_{0 \leq i \leq r}$ coincide from $m-1$. Thus, by applying the inductive hypothesis to $\{X_i\}_{0 \leq i \leq r}$ in place of $\{S_i\}_{0 \leq i \leq r}$ and $m-1$ in place of $m$, we conclude that the series $\{T_i\}_{0 \leq i \leq r}$ is associated with a strong $H$-pair series with $\mathfrak{F}$-head character $\chi$.

The uniqueness  of this strong $H$-pair series follows from  Lemma \ref{MH=UH}.
\end{proof}

\begin{thm}\label{5.4}
  Let $\mathfrak{F}$ be a saturated formation with $\mathfrak{N} \subseteq \mathfrak{F}$. Let $G$ be a solvable group, and let $\chi \in \irr G$. Then the following are equivalent:
  \begin{itemize}
    \item[(a)] $\chi$ is an $\mathfrak{F}'$-character of $G$;
    \item[(b)] $\chi$ is an $\mathfrak{F}$-head character of $G$.
  \end{itemize}
\end{thm}

\begin{proof}
First we show that (a) implies (b). Suppose $\chi$ is an $\mathfrak{F}'$-character of $G$. By Theorem \ref{caracterización descendente}, there exist characters $\varphi_i \in \irr{L_i}$ and $\theta_i \in \irr{K_i \mid \varphi_i}$ for $i=0,\ldots,m-1$, extending to $L_iH$ and $K_iH$, respectively, such that  
$$
(1,1) = (K_m, \theta_m) \triangleleft (L_{m-1}, \varphi_{m-1}) \triangleleft (K_{m-1}, \theta_{m-1}) \triangleleft \cdots \triangleleft (L_0, \varphi_0) \triangleleft (K_0, \theta_0) \triangleleft (G, \chi),
$$  
where 
$K_0 = G^{\mathfrak{F}}$, $L_0 = K_0'$, $K_i = (L_{i-1}H)^{\mathfrak{F}}$, and $L_i = K_i'$. Moreover, these characters satisfy $(\varphi_i)_{K_{i+1}}=\theta_{i+1}$ for all $i$.

From this series we can obtain an $H$-composition series. We consider now an $H$-composition factor $U/V$ of $G$. If $L_i \leq V < U \leq K_i$, then $(UH, U, V)$ satisfies the Navarro condition with respect to $\mathfrak{F}$. By applying Theorem \ref{key} to each $H$-composition factor between $L_i$ and $K_i$, we obtain $\alpha \in \irrh V$ and $\beta \in \irrh U$ such that they extend irreducibly to $VH$ and $UH$, respectively, and with $(V, \alpha) \lhd (U, \beta)$. 

If instead $K_{i+1} \leq V < U \leq L_i$, we take $\alpha = (\varphi_i)_V \in \irrh V$ and $\beta = (\varphi_i)_U$. These characters clearly extend irreducibly to $UH=VH=K_{i+1}H=L_iH$, and again we have $(V, \alpha) \lhd (U, \beta)$. 

In this way, we construct a strong pair series with $\chi$ as the $\mathfrak{F}$-head character.

\medskip

Now, we show that (b) implies (a). Let $\{S_j\}_{0 \leq j \leq r}$ be an $H$-composition series obtained from a refinement of the series  
$$
G \triangleright K_0 \triangleright L_0 \triangleright \cdots \triangleright K_{m-1} \triangleright L_{m-1} \triangleright K_m = 1
$$  
where 
$K_0 = G^{\mathfrak{F}}$, $L_0 = K_0'$, $K_i = (L_{i-1}H)^{\mathfrak{F}}$, and $L_i = K_i'$.
Since $\chi$ is an $\mathfrak{F}$-head character of $G$, by Theorem \ref{eslabon} we know that $\{S_j\}_{0 \leq j \leq r}$ is associated with a strong $H$-pair series $\{(S_j,\alpha_j)\}_{0 \leq j \leq r}$ such that $\alpha_r=\chi$. Moreover, by repeatedly applying Theorem \ref{MH=UH}, we have $(\alpha_{k+d})_{S_k}=\alpha_k$ whenever $S_k=K_{i+1}$ and $S_{k+d}=L_i$ for some $i \geq 0$, or $S_k=K_0$ and $S_{k+d}=G$. Therefore, by Theorem \ref{caracterización descendente}, $\chi$ is an $\mathfrak{F}'$-character.
\end{proof}

\section{Theorem A} \label{Sección restricciones}

The purpose of this section is to study how $\mathfrak{F}'$-characters of $G$ behave with respect to restriction, when $G$ is a solvable group and $\mathfrak{F}$ is a saturated formation. If $\mathfrak{N} \subseteq \mathfrak{F}$, there is a relationship between the $\mathfrak{F}'$-characters of $G$ and those of $NH$, where $N$ is a normal subgroup of $G$ and $H$ is an $\mathfrak{F}$-projector of $G$.

\begin{proposition}\label{6.1}
Let $\mathfrak{F}$ be a saturated formation such that $\mathfrak{N} \subseteq \mathfrak{F}$, and let $G$ be a solvable group. Suppose that $H$ is an $\mathfrak{F}$-projector of $G$, and assume that $G$ has odd order whenever $\mathfrak{N} \neq \mathfrak{F}$.
Let $K=G^\mathfrak{F}$ be the $\mathfrak{F}$-residual of $G$. Suppose that $M \lhd G$ is such that $K/M$ is a chief factor of $G$.
\begin{itemize}
    \item[(a)] If $\chi \in \irrf G$, then  there exists $\mu \in \irrf{MH}$ lying under $\chi$.
    \item[(b)] If $\mu \in \irrf{MH}$, then there exists $\chi \in \irrf G$ lying over $\mu$.
\end{itemize}
\end{proposition}
\begin{proof}

Let $\chi \in \irrf{G}$. The $\mathfrak{F}$-residual  of $G/M$ is $K/M$, which is abelian, and is therefore complemented by $HM/M$ (by Theorem \ref{residual}). 
That is, $(G,K,M)$ satisfies the Navarro condition with respect to $\mathfrak{F}$ and, by Theorem \ref{cor1}, we have that $\chi_{MH}$ has an irreducible constituent $\mu$ that restricts irreducibly to $M$.

We have that $\varphi = \mu_M$ is the unique $H$-invariant constituent of $(\chi_K)_M$ (by Theorem \ref{key}). Now, we can consider the following strong $H$-pair series, which passes through $K$ and $M$:
$$(1,1)=(S_0,\varphi_0) \triangleleft \cdots \triangleleft (M,\varphi)=(S_m,\varphi_m)\triangleleft (K, \chi_K)=(S_{m+1},\varphi_{m+1}) \triangleleft \cdots \triangleleft (S_r,\varphi_r)=(G,\chi)\,.$$
From Remark \ref{remark 3} and the fact that $\mu$ extends $\varphi$, it follows that $\mu \in \irrf{MH}$, and this proves (a).

Conversely, let $\mu \in \irrf{MH}$. Again by Theorem \ref{cor1}, there exists $\chi \in \irr G$ lying over $\mu$ that restricts irreducibly to $K$, with $\theta = \chi_K$ being the unique irreducible $H$-invariant constituent of $\varphi^K$ where $\varphi = \mu_M$ (by Theorem \ref{key}).

Now, let  $$(1,1)=(S_0,\varphi_0) \triangleleft \cdots \triangleleft (M,\mu)=(S_m,\varphi_m)\triangleleft \cdots \triangleleft (S_l, \varphi_l)=(MH, \mu)$$ be a strong $H$-pair series for $MH$ with $\mu$ as $\mathfrak{F}$-head character. We can find an $H$-composition series for $G$ that passes through $M$ and $K$:
$$1=T_0 \lhd \cdots \lhd T_t=M \lhd T_{t+1} = K \lhd \cdots \lhd T_r=G.$$

We  may then construct a  strong $H$-pair series for $G$ having $\chi$ as the $\mathfrak{F}$-head character and passing through $K$ and $M$ as follows: 
$$(1,1)=(S_0,\varphi_0) \triangleleft \cdots \triangleleft (M,\mu)=(S_m,\varphi_m)\triangleleft (K, \chi_K)=(T_{t+1},\chi_{T_{t+1}}) \triangleleft \cdots \triangleleft (T_r,\chi_{T_r})=(G,\chi)\,.$$ Therefore $\chi \in \irrf G$, and this proves (b).
\end{proof}

This allows us to prove a generalization of Theorem A(b) to saturated formations containing $\mathfrak{N}$, as presented below. Observe that, since this result depends on Proposition \ref{6.1} (and thus Theorem \ref{cor1}), the hypothesis that the group is of odd order whenever $\mathfrak{N}\neq \mathfrak{F}$ is still necessary. The previously mentioned group \texttt{SmallGroup(48,28)} serves as a counterexample.

\begin{thm}\label{NC} Let $\mathfrak{F}$ be a saturated formation such that $\mathfrak{N} \subseteq \mathfrak{F}$, and let $G$ be a solvable group. Suppose $H$ is an $\mathfrak{F}$-projector of $G$, and assume that $G$ has odd order whenever $\mathfrak{N} \neq \mathfrak{F}$. 
If $\chi\in \irrf G$ and $N$ is a normal subgroup of  $G$, then the restriction of $\chi$ to $NH$ has an irreducible constituent that is an $\mathfrak{F}'$-character of $NH$.
\end{thm}

\begin{proof}
    We argue by induction on $|G|$. Let $K$ be the $\mathfrak{F}$-residual of $G$.  
If $K \leq N$, then $NH = KH = G$, and the claim follows immediately.  

We may assume that   $K \cap N < K$, and choose $M \lhd G$ with $K \cap N \leq M < K$ such that 
$K/M$ is a chief factor of $G$. Since $\chi \in \irrf G$, the restriction 
$\chi_{MH}$ has an irreducible constituent $\mu \in \irrf{MH}$ by Lemma~\ref{6.1}.  

By Remark~\ref{caja}, $K \cap MH = M$ because $K/M$ is abelian. Hence $MH < KH$, otherwise 
$M = K$, a contradiction. 

We claim that $MH = NMH$. Set $\overline{G} = G/N$ and, adopting the bar convention, 
$\overline{K}$ is the $\mathfrak{F}$-residual of $\overline{G}$, and 
$\overline{K}/\overline{M}$ is abelian. Hence, by Proposition~\ref{basics},  
$$\overline{K} \cap \overline{M}\,\overline{H} = \overline{M},$$  
which implies $KN \cap MNH = MN$.  

Assume that $MNH = G$. Then $KN = MN$, and then
$$K = K \cap KN = K \cap MN = M(K \cap N)=M,$$
a contradiction. Thus,  $MNH < G$.  

Observe that $M \leq MNH \cap K \trianglelefteq G$ and $MNH \cap K < K$.
Since $K/M$ is a chief factor of $G$, it follows that  
$$MNH \cap K = M.$$
 
Hence,  
$$MH = (MNH \cap K)H = MNH\cap KH=MNH, $$  
proving the claim. 

 Now, $\mu \in \irrf{MH}$ and $N \trianglelefteq MH$. By the induction hypothesis, 
$\mu_{NH}$ has an irreducible constituent $\psi$ which is an $\mathfrak{F}'$-character of $NH$.  
As $\mu$ is a constituent of $\chi_{MH}$, we conclude that $\psi$ is a constituent of $\chi_{NH}$.  
This completes the proof.   \end{proof}

Consider a normal subgroup $N$ of $G$ and an $\mathfrak{F}$-projector $H$ of $G$, and suppose that either $\mathfrak{F}=\mathfrak{N}$ or $|G|$ is odd. Using the fact that $\mathfrak{F}'$-characters of $NH$ restrict irreducibly to $(NH)^\mathfrak{F}$ (by Theorem \ref{caracterización descendente}(a)) and that, by definition of the $\mathfrak{F}$-residual, $(NH)^\mathfrak{F}$ is contained in $N$, it follows directly from Theorem \ref{NC} that the restriction $\chi_N$ of any $\mathfrak{F}'$-character $\chi$ of $G$ will have an $H$-invariant constituent. However, it is possible to prove this result independently, even for groups of even order and saturated formations strictly containing $\mathfrak{N}$, as we state in Proposition \ref{prop B} below.

\begin{proposition} \label{prop B} 
Let $\mathfrak{F}$ be a saturated formation such that $\mathfrak{N} \subseteq \mathfrak{F}$, and let $G$ be a solvable group. Suppose $H$ is an $\mathfrak{F}$-projector of $G$.
If $\chi \in \irrf G$ and $N$ is a normal subgroup of $G$, then $\chi_N$ has a unique  $H$-invariant constituent.
\end{proposition} 

\begin{proof} We can find an $H$-composition series that passes through $G$ via $N$. Consider
$$1=S_0 \lhd \cdots \lhd S_m = N \lhd \cdots \lhd S_r=G\,.$$
By Theorem \ref{eslabon} and Theorem \ref{5.4}, this series is associated with a strong $H$-pair series $\{(S_i,\theta_i)\}_{0 \leq i \leq r}$ such that $\theta_r=\chi$. In particular, $\theta_m$ is a constituent of $\chi_N$ which extends to $NH$. Therefore,  it is $H$-invariant, which proves the existence of an $H$-invariant constituent. 

Uniqueness follows from an argument analogous to Isaacs' one in Lemma 3.5 of \cite{Isaacs}, which we adapt here to our context. Suppose $\alpha,\beta$ are $H$-invariant constituents of $\chi_N$. Then we may write $\beta=\alpha^g$ for some $g \in G$. We write $T=I_G(\alpha)$ the inertia subgroup. Then $H \leq T$ since $\alpha$ is $H$-invariant, and $H^g\leq T^g$. Since $\beta$ is also $H$-invariant and $I_G(\beta)=T^g$, we also have $H \leq T^g$. Now, by Proposition \ref{basics}(a) and since $\mathfrak{F}$-projectors form a conjugacy class, both $H$ and $H^g$ are $\mathfrak{F}$-projectors of $T$, hence $H=H^{gx}$ for some $x \in T^g$ by Proposition \ref{basics}(c). Then $gx \in \norm G H = H \leq T$, and thus $\alpha=\alpha^{gx}=\beta^x=\beta$, since $x \in T^g=I_G(\beta)$.
\end{proof}

We are now able to prove the results stated in the introduction. Considering $\mathfrak{F}=\mathfrak{N}$ yields Theorem A as it was formulated.

\begin{thm} \label{th A general}
Let $\mathfrak{F}$ be a saturated formation such that $\mathfrak{N} \subseteq \mathfrak{F}$, and let $G$ be a  solvable group. Suppose $H$ is an $\mathfrak{F}$-projector of $G$. Let $N$ be any normal subgroup of $G$. Then the following hold.
\begin{itemize}
\item[(a)] The restriction $\chi_N$ contains a unique $H$-invariant irreducible constituent $\theta$.
\item[(b)] We have that $\theta(1)$ divides $\chi(1)$ and that $\chi(1)/\theta(1)$ divides $|G:NH|$.
\item[(c)] If, moreover, $G$ has odd order whenever $\mathfrak{N}\neq \mathfrak{F}$, then the restriction $\chi_{NH}$ contains an $\mathfrak{F}'$-character $\gamma$ of $NH$. Furthermore, $\gamma_N=\theta$. Hence, any other $\mathfrak{F}'$-character of $NH$ contained in $\chi_{NH}$ is of the form $\lambda \gamma$, for some linear $\lambda \in \irr{NH/N}$.
\end{itemize}
\end{thm}

\begin{proof}
Assertion (a) is given by Proposition \ref{prop B}. Furthermore, consider the strong $H$-pair series $\{(S_i,\theta_i)\}_{0 \leq i \leq r}$ passing through $S_m=N$ where $\theta_m=\theta$ and $\theta_r=m$. For each $1 \leq i \leq r$, we have that $S_i/S_{i-1}$ is an $H$-composition factor of $G$ and thus $S_i \cap S_{i-1}H$ is either $S_{i-1}$ or $S_{i}$. If $S_i \cap S_{i-1}H=S_i$, then $S_i H =S_{i-1}H$ and by Theorem \ref{MH=UH} we have $\theta_i(1)=\theta_{i-1}(1)$. If  $S_i \cap S_{i-1}H=S_{i-1}$, then, since $\theta_i$ lies over $\theta_{i-1}$, we have $\theta_i(1)=e \theta_{i-1}(1)$ for some integer $e$ dividing $|S_i : S_{i-1}|$. But $S_iH/S_i \cong S_{i-1}H/S_{i-1}$ and thus $|S_i H : S_{i-1}H| = |S_i : S_{i-1}|$. By reasoning iteratively, it follows that $\theta(1)$ divides $\chi(1)$ and $\chi(1) / \theta(1)$ divides $|G:NH|$, and (b) is proved. 

Suppose for (c) that $|G|$ is odd whenever $\mathfrak{F} \neq \mathfrak{N}$. By Theorem \ref{NC}, $\chi_{NH}$ contains at least some $\gamma \in \irrf{NH}$ as a constituent. Since $(NH)^\mathfrak{F}\leq N$, we know that any $\mathfrak{F}'$-character of $NH$ restricts irreducibly to $N$. In particular, if $\delta$ is any $\mathfrak{F}'$-character of $NH$ lying under $\chi$, then $\delta_N$ is an $H$-invariant irreducible constituent of $\chi_N$ and thus is $\theta$ by uniqueness. That is, $\gamma_N=\theta$ and any other $\delta \in \irrf {NH}$ lying under $\chi$ is also an extension of $\theta$ to $NH$. By Gallagher's theorem, it follows that every such $\delta$
 is of the form $\lambda \gamma$ for some $\lambda \in \text{Lin}(NH/H)$. We may observe that, under the conditions of (c), we have $\gamma(1)=\theta(1)$ and thus $\chi(1)/\gamma(1)$ divides $|G:H|$.
\end{proof}

\section{Theorem B}

In this final section, we prove Theorem B in the more general context of saturated formations containing all nilpotent groups. This is formulated in Theorem \ref{kernels}. We first present some preliminary results.

\begin{lemma} \label{lema aux ker}
Let $\mathfrak{F}$ be a saturated formation with $\mathfrak{N} \subseteq \mathfrak{F}$. Let $G$ be a solvable group, and let $H$ be an $\mathfrak{F}$-projector of $G$. Let $M,T$ be $H$-invariant subgroups of $G$ such that $T \lhd M$ and $M/T$ is $H$-simple. Suppose that $M \cap H \subseteq H'$. If $\varphi$ is an irreducible $H$-invariant character of $M$ with $T$ in its kernel that extends to $MH$, then $\varphi$ is the trivial character.
\end{lemma}

\begin{proof}
Set $U=T(H \cap M)$. Since $M/T$ is $H$-simple, it is abelian, and thus $M/U$ is also abelian. Notice too that $H$ is an  $\mathfrak{F}$-projector of $MH$. Let $\varphi \in \irrh M$ that extends to $MH$ and has $T$ in its kernel. By applying Theorem \ref{key}(a) to $(MH,M,U)$, we know that there exists a unique $\xi \in \irrh {U}$ lying under $\varphi$. Moreover, $\xi$ extends to some $\eta \in \irr{TH}$ by Theorem \ref{key}(c).

Notice that, since $M/T$ is abelian, $\varphi$ is linear, and thus so are $\xi$ and $\eta$. Then $\eta$ can be seen as a character in $\text{Lin}(TH/T)$ with $H'T/T$ in its kernel. Thus, $U \subseteq H'T \subseteq \ker \eta$. Hence, $\xi=\eta_{ U}=1_{U}$. Since $\varphi$ is by Theorem \ref{key} (a) the unique $H$-invariant character lying over $\xi$, it must be $\varphi=1_M$.
\end{proof}

\begin{lemma}\label{ker} Let $\mathfrak{F}$ be a saturated formation with $\mathfrak{N} \subseteq \mathfrak{F}$. Let $G$ be a solvable group, and let $H$ be an $\mathfrak{F}$-projector of $G$. Suppose that  $N$ is a normal subgroup of $G$ such that $N\cap H\sbs H'$. Then $N\sbs \ker \chi$  for all $\chi \in \irrf G$. 

\end{lemma}

\begin{proof} Let $\chi \in \irrf G$. By Theorem \ref{5.4} and Theorem \ref{eslabon}, there exists a strong $H$-pair series of the form
$$(1,1)=(S_0,\varphi_0) \triangleleft \cdots \triangleleft (S_m,\varphi_m)=(N,\varphi_m) \triangleleft \cdots \triangleleft (S_r,\varphi_r)=(G,\chi).$$
We have that $S_i$ and $S_{i-1}$ are $H$-invariant, $S_i/S_{i-1}$ is $H$-simple and $S_i \cap H \subseteq H'$ for $i=1,\hdots,m$. Moreover, $\varphi_0=1_{S_0}$. Then, by repeatedly applying Lemma \ref{lema aux ker}, we deduce that $\varphi_i=1_{S_i}$ for every $0 \leq i \leq m$. In particular, $\varphi_m = 1_N$. Since $N \lhd G$ and $\varphi_m$ is a constituent of $\chi_N$, we have $N \subseteq \ker \chi$.
\end{proof}

If $N$ is a normal subgroup of $G$, every irreducible character $\tilde{\chi}$ of the quotient group $G/N$ can be regarded, by inflation, as a character of $G$ whose kernel contains $N$. In this case, we denote the corresponding character in $\irr G$ by $\chi$, where $N \subseteq \ker\chi$, and vice versa.

\medskip

\begin{lemma} Let $\mathfrak{F}$ be a saturated formation with $\mathfrak{N} \subseteq \mathfrak{F}$.  Suppose that  $N$ is a normal subgroup  of a solvable group $G$. 
\begin{itemize}
    \item[(a)] If $\chi \in \irrf G$ such that $N \subseteq \ker \chi$, then $\tilde \chi \in \irrf{G/N}$.
    \item[(b)] If $\tilde{\chi} \in \irrf{G/N}$, then $\chi \in \irrf G$.
\end{itemize}
\end{lemma}

\begin{proof} Let  $H$ be an $\mathfrak{F}$-projector of $G$. First suppose that $\chi \in \irrf G$ is such that $N \subseteq \ker \chi$. 
By Theorem \ref{5.4}, there exists a strong $H$-pair series for $G$ 
$$(1,1) = (S_0,\varphi_0) \nor (S_1, \varphi_1)\nor \cdots\nor (S_m, \varphi_m)=(N, \varphi_m) \nor\cdots\nor (S_r, \varphi_r)=(G, \chi)\,.$$

Since $S_m \subseteq \ker \chi$, and $\varphi_m$ lies under $\chi$, we have that $\varphi_m=1_{N}$ is the trivial character. Thus, $\varphi_i=1_{S_i}$ for every $0 \leq i \leq m$. Now, let $\overline G =G/N$. We have that $\overline H =HN/N$ is an $\mathfrak{F}$-projector of $\overline G$.  Adopting the bar convention for the quotients and the tilde convention for their irreducible characters, we can consider the following $\overline H$-pair series for $\overline{G}$,
$$ (\overline 1, \tilde 1) = (\overline{S_m}, \tilde{\varphi}_m) \nor \cdots \nor (\overline {S_i}, \tilde{\varphi_i})\nor \cdots \nor (\overline{S_r}, \tilde{\varphi}_r)=(\overline G, \tilde{\chi})$$
which is a strong $\overline H$-pair series as it is not difficult to verify. Hence, $\tilde\chi $ is an $\mathfrak{F}'$-character of $\overline G$ and we have proved (a).

 Now, suppose that $\tilde{\chi} \in \irrf{G/N}$, then there is a strong $\overline H$-pair series for $\overline{G}$
 $$ (\overline 1, \tilde 1) = (\overline{S_0}, \tilde{\varphi_0}) \triangleleft \cdots \triangleleft (\overline{S_i}, \tilde{\varphi_i})\triangleleft \cdots \triangleleft (\overline{S_r}, \tilde{\varphi_r})=(\overline G, \tilde{\chi})\, .$$
Let
$$1= R_0  \triangleleft R_1 \triangleleft \cdots \triangleleft R_s= N = S_0$$
be an $H$-composition series of $N$. It suffices to consider the following series, viewing the characters of the quotient as characters of the group with $N$ in the kernel, as we have mentioned previously,
$$ (1,1)= (R_0, 1)\nor \cdots \nor (R_s,1_{R_s})=(N, 1_N)= ({S_0}, {\varphi_0})\nor  \cdots \nor ({S_r}, {\varphi_r})=( G, {\chi}),$$
which again may be checked to be a strong $H$-pair series for $G$, proving the statement (b).
\end{proof}

\begin{lemma}
Let $\mathfrak{F}$ be a saturated formation with $\mathfrak{N} \subseteq \mathfrak{F}$. Let $G$ be a solvable group. Suppose that $N$ is a normal subgroup of $G$ contained in $\ker \chi$, for all $\chi \in \irrf G$.
Then $$|\irrf  G | =|\irrf {G/N}|\,.$$
Moreover, $N \cap H \subseteq H'$ for any $\mathfrak{F}$-projector $H$ of $G$.
\end{lemma}

\begin{proof} The previous lemma gives us an injection $\irrf G \to \irrf {G/N}$. Hence
$$|H/H'|=|\irrf G| \leq |\irrf {G/N}|= \left| \frac{HN/N}{(HN/N)'}\right|=\left| \frac{H/(H \cap N)}{H'/(H' \cap N)} \right| \leq |H/H'|,$$
where $H$ is any $\mathfrak{F}$-projector of $G$. Thus, we have the first part of the lemma. We conclude for the second part by noting that $H \cap N = H' \cap N \subseteq H'$.
\end{proof}

\begin{thm} \label{kernels} Let $\mathfrak{F}$ be a saturated formation with $\mathfrak{N} \subseteq \mathfrak{F}$. Let $G$ be a solvable group, and let $H$ be an $\mathfrak{F}$-projector of $G$. Then
    $$
     \bigcap_ {\chi \in \irrf G} \ker \chi
    $$
is the largest normal subgroup $M$ of $G$ satisfying $M \cap H \subseteq H'$. \end{thm}

\begin{proof} Let $M = \bigcap_{\chi \in \irrf G} \ker \chi$. It is clear that $M \cap H \subseteq H'$ from the previous lemma.

Suppose $N$ and $K$ are two normal subgroups of $G$ such that $N \cap H \subseteq H'$ and $K \cap H \subseteq H'$. It follows  that $NK \cap H = (N \cap H)(K \cap H) \subseteq H'$, by Lemma \ref{hup}. Thus, there exists a maximal subgroup $N$ satisfying $N \cap H \subseteq H'$. By Lemma \ref{ker}, if $N$ is a normal subgroup of $G$ such that $N \cap H \subseteq H'$, then $N \subseteq M$. This completes the proof of the theorem.
\end{proof}

The above theorem is the exact analog of an unpublished result by Navarro, which we present here with his permission. Recall that, for a prime $p$ and a solvable group $G$, the Sylow $p$-subgroups of $G$ and the characters of $G$ of $p'$-degree correspond respectively to the $\mathfrak{F}$-projectors and the $\mathfrak{F}'$-characters of $G$ when $\mathfrak{F}$ is the formation of $p$-groups. Since this is a saturated formation not containing $\mathfrak{N}$, we cannot  apply Theorem \ref{kernels} directly.

\begin{thm}[Navarro] Let $G$ be a finite group, let $p$ be a prime, and let $P \in \syl p G $. Then
$$
K = \bigcap_{\chi \in {\rm Irr}_{p'}(G)} \ker\chi
$$
is the largest normal subgroup $K$ of $G$ satisfying $\norm K P \subseteq P'$. 
\end{thm}
\begin{proof}
If $N, M$ are normal subgroups of $G$, by Lemma 2.1 of \cite{IN}, we have that 
$$\norm {NM} P = \norm M P \norm N P \,.$$
Therefore  there exists a largest normal subgroup $N$ of $G$ such that $\norm N P$ is contained in $P'$.  Since $N \cap P = \norm{N}{P} \subset P'$, by Tate's theorem (Theorem 6.31 of \cite{Is}), it follows that $\Oh{p}{PN} \cap P = \Oh{p}{P} = 1$.  Therefore, $PN$ has a normal $p$-complement, and  $N$ also has one. Let $W$ be the normal $p$-complement of $N$. Then, we have $\cent{W}{P}=1$ because $\cent W P \subseteq \norm NP \cap W\subseteq P'\cap W=1$.     Let $\chi \in \irrp G$, and let $\nu \in \irr W$ be a $P$-invariant irreducible constituent of the restriction $\chi_W$, which exists beacause $\chi$ has $p'$-degree. By Glauberman’s correspondence (Theorem 2.9 of \cite{N2}), we have that $\nu$ is the trivial character of $W$. Hence, $W \subseteq \ker{\chi}$. So, working in $G/W$ we may assume that $N$ is a $p$-group contained in $P'$. Now, let $\epsilon \in \irr P$ be a linear constituent of $\chi_P$. Since $1_N$ is a constituent of $\epsilon_N$, we have that $N \subseteq \ker \chi$. 

Let 
$$  K = \bigcap_{\chi \in \irrp G} \ker\chi\,. $$  
It only remains to prove that $\norm{K}{P} \subseteq P'$.  
 Using Theorem 7.7 of \cite{N2}, $K$ has a normal $p$-complement $Y$. If $\delta \in \irr Y$ is $P$-invariant, then there is an extension $\eta$  of $\delta$ to $Y P$ (by Corollary 6.28 of \cite{Is}), and therefore $\delta$ lies under some irreducible character of $G$ of $p'$-degree (using that $\eta^G$ has $p'$-degree). Thus $\delta$ is the trivial character of $Y$, and hence $\cent YP = 1$ by Glauberman's correspondence. Thus $\norm K P = P \cap K$. We may assume that $K$ is a $p$-group. Suppose that $K$ is not contained in $P'$. Then some linear $\lambda \in \irr P$ does not contain $K$ in its kernel. By inducing up to $G$, we obtain some $\chi \in \text{Irr}_{p'}(G)$ over $\lambda$ and therefore a contradiction.
\end{proof}

\end{document}